\documentclass[12pt]{amsart}

\usepackage{amsmath, amssymb}
\usepackage{amsfonts,latexsym,bbm}
\usepackage{amsthm,amscd}

\newtheorem{thm}{Theorem}
\newtheorem{lem}[thm]{Lemma}

\newtheorem{cor}[thm]{Corollary}

\theoremstyle{definition}

\theoremstyle{remark}
\newtheorem{rem}[thm]{Remark}

\numberwithin{equation}{section}

\allowdisplaybreaks

\newcommand{\X}{\mathbf X}
\newcommand{\A}{\mathbf A}
\newcommand{\V}{\mathbf V}
\newcommand{\schwach}{\stackrel{\mathcal D}{\longrightarrow}}
\newcommand{\vage}{\mbox{$\, \stackrel{v}{\longrightarrow} \,$}}
\newcommand{\leb}{\mbox{$\lambda\hspace{-.12cm}\lambda$}}

\begin{document}
\sloppy
\title[Scaling limits of coupled continuous time random walks ]{Scaling limits of coupled continuous time random walks and residual order statistics through marked point processes}

\author{Adam Barczyk}
\address{Mathematical Institute, Heinrich-Heine-University D\"usseldorf, D-40225 D\"usseldorf, Germany}
\email{barczyk\@@{}math.uni-duesseldorf.de}

\author{Peter Kern}
\address{Mathematical Institute, Heinrich-Heine-University D\"usseldorf, D-40225 D\"usseldorf, Germany}
\email{kern\@@{}math.uni-duesseldorf.de}

\date{\today}

\begin{abstract}
A continuous time random walk (CTRW) is a random walk in which both spatial changes represented by jumps and 
waiting times between the jumps are random. The CTRW is coupled if a jump and its preceding or following waiting time are dependent random 
variables (r.v.), respectively. 
The aim of this paper is to explain the occurrence of different limit processes for CTRWs with forward- or backward-coupling in 
Straka and Henry (2011) using marked point processes. We also establish a series representation for the different limits. The methods used also allow us to solve an 
open problem concerning residual order statistics by LePage (1981).
\end{abstract}



\maketitle

\baselineskip=18pt

\section{Introduction}
\label{Introduction}

Two i.i.d.\ sequences of $\mathbb R^+$-valued waiting times $(J_n)_{n \in \mathbb N}$ and of $\mathbb R^d$-valued jumps 
$(\X_n)_{n \in \mathbb N}$ yield two versions of a CTRW by 
\begin{align*}
S_{N_t} := \sum_{k=1}^{N_t} \X_k, ~S_{N_t +1} := \sum_{k=1}^{N_t+1} \X_k,
\end{align*}
where $N_t := \max\{n \in \mathbb N_0: \sum_{k=1}^n J_k \leq t\}$ is the number of jumps up to time $t$. 
The CTRW is coupled if the sequences $(J_n)_{n \in \mathbb N}$ and 
$(\X_n)_{n \in \mathbb N}$ are dependent. Typically we assume, that the sequence $(J_n,\X_n)_{n \in \mathbb N}$ is i.i.d. 
with unknown dependence between the waiting time $J_n$ and the jump $\X_n$ for fixed $n\in \mathbb N$. 
Using this dependence structure $S_{N_t}$ is called backward-coupled CTRW whereas $S_{N_t+1}$ is called forward-coupled CTRW. 
Both processes represent the position of a jumper at time $t$, but in the backward-coupled case the particle first waits for a 
time $J_1$ before jumping to $\X_1$, whereas in the forward-coupled case the particle jumps to $\X_1$ at time $t=0$ and then 
waits for a time $J_1$ and so on.
CTRW processes were 
introduced in \cite{montroll:1965} to study random walks on 
a lattice and have been studied intensively over the past few decades. Today there is a wide field of possible applications for CTRWs. 
They are used in physics to model phenomena of 
anomalous diffusion \cite{shlesinger:1982,MetKla:2000}. The jumps can also represent movements of an ensemble of particles being transported 
over the earth surface in geophysics \cite{schumer:2009} or represent log-returns in finance \cite{scalas:2004}. 
A comprehensive study of limit theorems for coupled CTRWs has been initiated in \cite{becker-kern:2004} covering previously known special models \cite{shlesinger:1982,klafter:1987,kotulski:1995} from physics. Coupling methods different from forward-/backward-coupling can be found in \cite{CheHofSok:2009,MagMetSzcZeb}, where certain correlations between the waiting times, respectively between the jumps were introduced.
A similar approach in a more general setting appears in \cite{jara:2011}, where coupling is introduced through a Markov chain $(Y_n)_{n\in\mathbb N}$ and waiting times, respectively jumps are modelled by $(J_n=\tau(Y_n),X_n=V(Y_n))$ for some measurable functions $\tau,V$ with values in $(0,\infty)$, respectively $\mathbb R$. Note that these couplings in general do not fulfill our i.i.d.\ assumption on the sequence $(J_n,\X_n)_{n \in \mathbb N}$.

The limiting distributions of forward- and backward-coupled CTRWs have been investigated by Straka and Henry \cite{henry:2010} using a continuous mapping approach on the space of their sample paths. Straka and Henry prove that the limiting processes of coupled CTRWs in general differ when waiting times precede or follow jumps, 
respectively. Also the differences between the properties of these processes are not marginal, cf. \cite{jurlewicz:2010}. 
Unfortunately, neither the continuous mapping 
approach used in \cite{henry:2010} nor the methods used in \cite{jurlewicz:2010} are adequate to point out why 
different scaling limits occur from a mathematical point of view. So a new approach to fill this gap is made using marked point processes here. Another explanation for the different scaling limits of forward- and backward-coupled CTRW models can be found when introducing a random clustering procedure as in \cite{WerJurMagWerTrz:2010,JurMeeSch,StaWer:2010}. A subsequent clustering of $M_n$ uncoupled i.i.d.\ waiting times and jumps $(J_i,\X_i)$, where $(M_n)_{n\in\mathbb N}$ is a sequence of positive integer valued i.i.d.\ random variables belonging to the domain of attraction of a $\gamma$-stable random variable with $0<\gamma<1$, leads to models with forward- and backward-coupling, regarding to the number of cluster jumps that have been occasionally or fully considered at the observation time. The corresponding scaling limits are described by over- and undershooting subordination in \cite{WerJurMagWerTrz:2010,JurMeeSch,StaWer:2010}, where the over- and undershoot process is again the scaling limit of a tightly coupled CTRW with one-dimensional jumps $X_i=J_i$ equal to the corresponding waiting times. The special models of tightly coupled CTRW scaling limits have been extensively studied in \cite{kotulski:1995,bertoin:1996,SzcZeb:2012} in connection with generalized arc-sine laws. Our approach using marked point processes will illuminate the occurence of different scaling limits in more general situations than the case of a tightly coupled CTRW.

Defining the 
time of the $n$-th jump by $T_n := \sum_{k=1}^n J_k$ we first study the limit behavior of the point processes which arise 
by marking each jump time with its occuring jump, respectively, i.e. we analyze
\begin{align} \label{pp.einleitung}
 \sum_{k=1}^n \varepsilon_{(T_k,\X_k)}, ~\sum_{k=1}^n \varepsilon_{(T_{k-1},\X_k)}.
\end{align}
It turns out that only the jumps with large norm contribute to the limit distributions of (\ref{pp.einleitung}), as it is 
already known for real-valued partial sums which converge to an infinitely divisible r.v. without Gaussian part, cf. \cite{barczyk:2010} 
and references therein. The methods 
used also solve an open problem concerning the convergence of residual order statistics by LePage, 
cf. \cite{lepage:1981}, \cite{lepagezinn:1981}, \cite{scheffler:1998}. The scaling limits of the CTRWs can be determined 
by summing up the marks of the points in (\ref{pp.einleitung}) which have a jump time occuring before time $t$. Hence in the scaling limit of forward-coupled CTRWs an additional big jump occurs compared to its backward-coupled version, which illuminates the difference between the processes. This 
approach also provides a series representation for the different scaling limits which might be of interest for simulation purposes. 
Since the resulting limit processes are not L\'evy processes, no efficient simulation algorithm is known yet.

\section{Preliminaries}

Let $(J_n,\X_n)_{n \in \mathbb N}$ be an i.i.d. sequence of $\mathbb R^+ \times \mathbb R^d$-valued r.v. Assume that 
$(J_1,\X_1)$ belongs to the generalized domain of attraction (GDOA) of a r.v.\ $(D,\A)$, where $D$ is stable with 
index $\alpha \in (0,1)$ and $\A$ is full operator stable with index $E\in GL(\mathbb R^d)$ and without Gaussian part. Note that 
by Theorem 7.2.1 of \cite{meerschaert:2001} the real parts of the eigenvalues of $E$ are greater than $1/2$. 
By classic results \cite[Theorem 14.14]{kallenberg:1986} and \cite[Theorem 4.1]{meerschaert:2004} this implies 
that there exist regularly varying sequences $(b_n)_{n \in \mathbb N} \in RV_{1/\alpha}$ and $(A_n)_{n \in \mathbb N} \in RV_{-E}$ such that the 
convergence
\begin{align} \label{vorraussetzung}
 \left(b_n^{-1} \sum_{k=1}^{\lfloor nt \rfloor} J_k ,~
\sum_{k=1}^{\lfloor nt \rfloor} (A_n \X_k - \mathbb E(A_n \X_k \mathbf 1_{\Vert A_n \X_k \Vert \leq \tau})) \right)_{t\geq 0}
\schwach (D(t),\A(t))_{t \geq 0}
\end{align}
holds in $D([0,\infty),\mathbb R^+ \times \mathbb R^d)$ for any $\tau>0$ such that for the L\'evy measure $\eta$ of $\A$ and the sphere 
$\mathbb S^{d-1}_\tau = \{x \in \mathbb R^d : \Vert x \Vert = \tau \}$ we have 
$\eta(\mathbb S^{d-1}_\tau)=0$, 
i.e. the sphere $\mathbb S^{d-1}_\tau$ is a continuity set for $\eta$. Throughout this paper, spaces of c\`adl\`ag paths $D([0,\infty),\mathcal X)$, respectively $D([0,T],\mathcal X)$, are always equipped with the corresponding Skorokhod ${\rm J}_1$-topology. Here the process $D(\cdot)$ 
denotes an $\alpha$-stable subordinator and $\A(\cdot)$ denotes an operator L\'evy motion. 
Note that the drift term of the L\'evy process $\A$ depends on $\tau$. It is well known that we can choose $\tau = \infty$ 
if the real part of any eigenvalue of $E$ belongs to $(1/2,1)$, since then $\mathbb E(\X_1)$ exists. Moreover we can choose 
$\tau = 0$ if the real part of any eigenvalue of $E$ exceeds $1$. Due to the spectral decomposition in \cite{meerschaert:2001}, 
centering by truncated expectations in (\ref{vorraussetzung}) is only necessary if some eigenvalue of the exponent $E$ has real 
part equal to $1$.

We already stated, that only points with large norm contribute to the limit of the point processes in (\ref{pp.einleitung}). So 
we use a radial decomposition of the L\'evy measure
\begin{align*}
 \eta(A) = \int_{\mathbb S^{d-1}} \int_0^\infty \mathbf 1_{A}(xv) \widetilde \eta(dx,v) d\sigma(v)
\end{align*}
where $\sigma$ is a probability measure on the unit sphere $\mathbb S^{d-1}$ of $\mathbb R^d$ and 
$(\widetilde \eta(\cdot,v))_{v \in \mathbb S^{d-1}}$ is a weakly measurable family of L\'evy measures on $(0,\infty)$, cf. 
\cite{rosinski:2001}. We also define the right-continuous inverse of $\widetilde \eta(\cdot,v)$ by
\begin{align} \label{rechtsinverse}
 \widetilde \eta^\leftarrow (x,v) := \sup\{u>0 : \widetilde \eta ([u,\infty),v) \geq x\}.
\end{align}
With this notation we are able to give a series representation for the process $\A(\cdot)$ in $D([0,T],\mathbb R^d)$ for 
fixed $T>0$ by 
\begin{align} \label{reihendarst.}
 \lim_{\varepsilon \downarrow 0} \left( \sum_{T\cdot\tau_k\leq t} \left(\widetilde \eta^\leftarrow (T^{-1} \Gamma_k,\V_k)\V_k 
 \mathbf 1_{\eta^\leftarrow (T^{-1} \Gamma_k,\V_k) >\varepsilon}\right) 
 - t \int_{\varepsilon \leq \Vert x \Vert \leq \tau} x d\eta(x)\right),
\end{align}
cf. \cite{lepage:1981}, \cite{rosinski:2001}, where $\Gamma_n$ is the $n$-th partial sum of i.i.d. standard exponential r.v., 
$(\tau_n)_{n \in \mathbb N}$ denotes an i.i.d. sequence of uniformly $\mathcal U(0,1)$-distributed r.v. and 
$(\V_n)_{n \in \mathbb N}$ denotes an i.i.d. sequence with distribution $\sigma$, with $(\Gamma_n)_{n \in \mathbb N}$, 
$(\tau_n)_{n \in \mathbb N}$ and $(\V_n)_{n \in \mathbb N}$ being independent.

Now it is well known that for a triangular array of infinitesimal row-wise independent $\mathbb R^+$-valued r.v. 
$(Y_{k,n})_{1 \leq k \leq n}$, $n \in \mathbb N$, converging to an infinitely 
divisible r.v. $Y$ with associated L\'evy measure $\phi$, only the extremes contribute to the limit distribution, 
cf. \cite{barczyk:2010} and references therein. This result coincides with the convergence
\begin{align} \label{grund.konv.pp}
 \sum_{k=1}^n \varepsilon_{\left(\frac{k}{n},Y_{k,n}\right)} \schwach PRM(\leb \otimes \phi)
\end{align}
in $M_p([0,1] \times (0,\infty])$, the set of all point measures on $[0,1] \times (0,\infty]$, where $PRM(\leb \otimes \phi)$ denotes a Poisson random measure with mean measure 
$\leb \otimes \phi$. Furthermore, it is well known that $\sum_{k \in \mathbb N} \varepsilon_{(\tau_k,\phi^\leftarrow(\Gamma_k))}$ is also 
a representation of $PRM(\leb \otimes \phi)$ in $M_p([0,1] \times (0,\infty])$, where $\phi^\leftarrow$ denotes the 
right-sided inverse of $\phi$, cf. \cite{rosinski:2001}. This fact can be understood by sorting the points on the left-hand 
side in (\ref{grund.konv.pp})
\begin{align} \label{umsortierung}
 \sum_{k=1}^n \varepsilon_{\left(\frac{k}{n},Y_{k,n}\right)}
 = \sum_{k=1}^n \varepsilon_{\left(\frac{d_k}{n},Y_{n-k+1:n}\right)},
\end{align}
where $(Y_{1:n},\ldots,Y_{n:n})$ denotes the order statistics of $(Y_{1,n},\ldots,Y_{n,n})$ with corresponding antirank vector 
$(d_1,\ldots,d_n)$, i.e. the inverse permutation of the rank vector. Using Freedman's Lemma, 
cf. \cite{janssen:1990}, one can easily verify, that the convergence 
$(n^{-1}d_k)_{k \in \mathbb N} \schwach (\tau_k)_{k \in \mathbb N}$ holds in $[0,1]^\mathbb N$, where $d_k=0$ for $k>n$. 
Moreover, $(2.4)$ in \cite{barczyk:2010} gives us $Y_{n-k+1:n} \schwach \phi^\leftarrow(\Gamma_k)$. So the 
convergence in (\ref{grund.konv.pp}) can also be established by analyzing the convergence of the points 
$(n^{-1}d_k,Y_{n-k+1:n})_{1\leq k \leq n}$. This approach can also be applied to the point processes in (\ref{pp.einleitung}). 
As the r.v.\ $(\X_n)_{n \in \mathbb N}$ are $\mathbb R^d$-valued one cannot use traditional order statistics. LePage suggested 
to use a normwise sorting, cf. \cite{lepage:1981}. So for $x_1,\ldots,x_n\in \mathbb R^d$ we introduce the residual 
order statistics $x_{1:n},\ldots,x_{n:n}$ by $\Vert x_{1:n}\Vert \leq \ldots \leq \Vert x_{n:n}\Vert$.

\section{Convergence of residual order statistics}

The convergence of the normalized residual order statistics $A_n \X_{n-k+1:n}$ is still an open problem. LePage 
\cite{lepage:1981} conjectures, that a generalization of the one-dimensional case 
\begin{align} \label{vermutung.lepage}
 (A_n \X_{n-k+1:n})_{k\in \mathbb N} \schwach (\widetilde \eta^\leftarrow(\Gamma_k,\V_k)\V_k)_{k \in \mathbb N}
\end{align}
holds. As usual one sets $\X_{k:n}=0$, whenever $k\leq 0$ or $k>n$. This result has been proven in \cite{lepagezinn:1981} for 
the case that the limit process $\A(\cdot)$ is multivariate $\alpha$-stable. In this case the right-sided inverse 
$\widetilde \eta^\leftarrow(x,v)$ defined in (\ref{rechtsinverse}) is independent of $v \in \mathbb S^{d-1}$ as the projection 
of the L\'evy measure $\eta$ is the same for every direction, cf. Theorem 7.3.3 in \cite{meerschaert:2001}. Some years later 
a similar problem has been studied in \cite{hahn:1989} using a different norm $\Vert \cdot \Vert_H$ which respects the special structure 
of the operator $E$. In \cite{scheffler:1998} the operator semistable case has been studied, but the result (\ref{vermutung.lepage}) 
also could only be established in the special case, that $\widetilde \eta^\leftarrow(x,v)$ is independent of $v$, which concides 
with the multivariate $\alpha$-stable case. The author also supposed that the convergence (\ref{vermutung.lepage}) holds only in this case. 
We will show that the limit on the right-hand side in (\ref{vermutung.lepage}) has to be modified. The proof is based on the 
following lemma. 

\begin{lem} \label{kvgz.pp}
 Let $N_n = \sum_{k=1}^n \varepsilon_{\X_k^{(n)}},~n\in \mathbb N_0$, be a sequence of point processes in 
$M_p([-\infty,\infty]^d \backslash \mathbb K_\varepsilon^d)$, the set of all point measures on the space 
$[-\infty,\infty]^d \backslash \mathbb K_\varepsilon^d$, where 
$\mathbb K_\varepsilon^d := \{x \in \mathbb R^d : \Vert x \Vert \leq \varepsilon \}$ denotes the compact $\varepsilon$-ball 
in $\mathbb R^d$. Suppose $N_n \schwach N_0$. If 
$\varepsilon < \Vert \X_i^{(0)} \Vert<\infty$ holds for every $i \in \mathbb N$ almost surely (a.s.) and 
$\Vert \X_i^{(0)}(\omega)\Vert > \Vert \X_j^{(0)}(\omega)\Vert$ for all $1\leq i<j$ a.s. then the 
convergence 
\begin{align*}
 \left(\X_{n-k+1:n}\right)_{k \in \mathbb N} \schwach \left(\X_k^{(0)}\right)_{k \in \mathbb N}
\end{align*}
holds in $\left(\mathbb R^d \backslash \mathbb K_\varepsilon^d\right)^\mathbb N$.
\end{lem}

\begin{proof}
 The proof is based on a continuous mapping approach and Lemma 7.1 in \cite{resnick:2007}. Define 
$M \subset M_p([-\infty,\infty]^d \backslash \mathbb K_\varepsilon^d)$ by
\begin{eqnarray*}
   M := \Big\{ m : m=\sum_{k=1}^P \varepsilon_{x_k},~ \infty> \Vert x_1\Vert > \ldots > \Vert x_P\Vert>\varepsilon\Big\}.
\end{eqnarray*}
Now we show that the mapping 
\begin{align*}
& \pi_k: M_p([-\infty,\infty]^d \backslash \mathbb K_\varepsilon^d) \rightarrow [-\infty,\infty]^d \backslash \mathbb K_\varepsilon^d \\
& \pi_k \left(\sum_{i=1}^P \varepsilon_{x_i}\right) \mapsto x_{P-k+1:P},~x_{P-k+1:P}=0 \text{ for } k\leq 0 \text{ and } k>P
\end{align*}
is continuous in $m \in M$ for every $k \in \mathbb N$. Let 
$(m_n)_{n \in \mathbb N} \subset M_p([-\infty,\infty]^d \backslash \mathbb K_\varepsilon^d)$ be a sequence of point measures 
converging vaguely to a point measure 
$m_0 = \sum_{k=1}^P \varepsilon_{x_k^{(0)}} \in M_p([-\infty,\infty]^d \backslash \mathbb K_\varepsilon^d)$. Now choose $n$ 
sufficiently large so that all points of $m_n$ lie inside of $[-\infty,\infty]^d \backslash \mathbb K_\varepsilon^d$. 
By sorting the points of $m_n$ in descending order of their norm 
\begin{align*}
 m_n = \sum_{i=1}^P \varepsilon_{x_i^{(n)}},~\varepsilon<\Vert x_P^{(n)}\Vert \leq \ldots \leq \Vert x_1^{(n)}\Vert<\infty
\end{align*}
an application of Lemma 7.1 in \cite{resnick:2007} yields the convergence of the points
\begin{align} \label{kvgz.pkte}
 \left(x_1^{(n)},\ldots ,x_P^{(n)}\right) \longrightarrow \left(x_1^{(0)},\ldots ,x_P^{(0)}\right)
\end{align}
in $(\mathbb R^d\backslash \mathbb K_\varepsilon^d)^P$. Now by the definition of the mapping $\pi_k$ 
\begin{align*}
 & \pi_k(m_n) = x_k^{(n)},~\pi_k(m_0)=x_k^{(0)} \text{ for } 1 \leq k \leq P \\ 
 & \pi_k(m_n) = \pi_k(m_0)=0  \text{ for } k>P
\end{align*}
holds and $\pi_k$ is continuous by (\ref{kvgz.pkte}) for every $k \in \mathbb N$. An easy application of the 
continuous mapping theorem yields the desired result. 
\end{proof}

Lemma \ref{kvgz.pp} allows to identify the distribution of the limit of properly normalized residual order statistics. 

\begin{thm} \label{verm.lepage}
 For $k\in \mathbb N$, $T>0$ and $\omega \in \Omega$ define 
\begin{align} \label{def.res.antirang}
 \widehat d_k(\omega)= \arg \left(\max_{\substack{i\in \mathbb N \\ i\neq \widehat d_1(\omega),\ldots,\widehat d_{k-1}(\omega)}} 
\widetilde \eta^\leftarrow (T^{-1}\Gamma_i(\omega),\V_i(\omega)) \right)
\end{align}
as the argument of the $k$-th largest element of the set 
$\{\eta^\leftarrow (T^{-1}\Gamma_i(\omega),\V_i(\omega)), i \in \mathbb N  \}$. Then convergence of the residual order statistics
\begin{align} \label{resid.order}
 \left(A_n\X_{\lfloor nT\rfloor -k+1:\lfloor nT\rfloor}\right)_{k \in \mathbb N} 
 \schwach \left(\widetilde \eta^\leftarrow(T^{-1}\Gamma_{\widehat d_k},\V_{\widehat d_k})\V_{\widehat d_k} \right)
\end{align}
holds in $(\mathbb R^d)^\mathbb N$.
\end{thm}

\begin{rem} \label{remark.existenz}
Note that $\widehat d_k$ is well-defined for all $k \in \mathbb N$, since the number of elements in the set
$\{ i\in \mathbb N: \widetilde \eta^\leftarrow (T^{-1}\Gamma_i,\V_i)>\varepsilon \}$ is finite a.s. for all 
$\varepsilon >0$. Moreover theorem \ref{verm.lepage} does not contradict any of the results proven in \cite{hahn:1989}, 
\cite{lepagezinn:1981}, \cite{scheffler:1998}. If $\A(\cdot)$ is a multivariate $\alpha$-stable L\'evy process, the monotonicity of the mapping 
$x \mapsto \widetilde \eta^\leftarrow(x,v)$ yields $\widehat d_k(\omega) = k$ a.s. for all $k \in \mathbb N$.
\end{rem}

\begin{proof}[Proof of Theorem \ref{verm.lepage}]
 First we have to determine the limit of the truncated point process
\begin{align*}
 \sum_{k=1}^{\lfloor nT \rfloor} \varepsilon_{(A_n\X_k \mathbf 1_{\Vert A_n\X_k\Vert \geq \varepsilon})},
\end{align*}
where $\varepsilon>0$ has to be choosen such that $\eta(\mathbb S^{d-1}_\varepsilon)=0$ holds. 
By Theorem 3.2.2 in \cite{meerschaert:2001} assumption (\ref{vorraussetzung}) yields the vage convergence 
\begin{align} \label{konv.levymaß}
 \lfloor nT\rfloor \mathbb P\left(A_n\X_k \in \cdot \right) \vage T \cdot \eta(\cdot)
\end{align}
in $\mathbb R^d \backslash \{0\}$. Hence the convergence of the point processes 
\begin{align*} 
 \sum_{k=1}^{\lfloor nT\rfloor} \varepsilon_{A_n\X_k} \schwach PRM(T\cdot\eta)
\end{align*}
holds in $M_p([-\infty,\infty]^d \backslash\{0\})$. Now by \cite{rosinski:2001}
$\sum_{k\in \mathbb N} \varepsilon_{\widetilde \eta^\leftarrow(T^{-1}\Gamma_k,\V_k)\V_k} = PRM(T\cdot\eta)$.
Applying the a.s. continuous restriction functional
\begin{align} \label{einschr.abb}
 \pi':M_p([-\infty,\infty]^d\backslash \{0\})\rightarrow M_p([-\infty,\infty]^d\backslash \mathbb K_\varepsilon^d) 
,~m \mapsto m_{|(\mathbb K^d_\varepsilon)^\complement},
\end{align}
the continuous mapping theorem yields
\begin{align} \label{kvgz.eigeschr.pp}
 \sum_{k=1}^{\lfloor nT\rfloor} \varepsilon_{A_n \X_k \mathbf 1_{\Vert A_n \X_k \Vert \geq \varepsilon}}
 \schwach \sum_{k \in \mathbb N} \varepsilon_{\widetilde \eta^\leftarrow(T^{-1}\Gamma_k,\V_k)\V_k 
  \mathbf 1_{\widetilde \eta^\leftarrow(T^{-1}\Gamma_k,\V_k) \geq \varepsilon}}.
\end{align}
The continuity of $\pi'$ is proven in \cite{feigin:1996} for instance. Now the points of the point process on the right-hand side 
of (\ref{kvgz.eigeschr.pp}) have to be ordered in descending order of their norm 
\begin{align*}
\sum_{k \in \mathbb N} \varepsilon_{\widetilde \eta^\leftarrow(T^{-1}\Gamma_k,\V_k)\V_k 
  \mathbf 1_{\widetilde \eta^\leftarrow(T^{-1}\Gamma_k,\V_k) \geq \varepsilon}}
= \sum_{k \in \mathbb N} \varepsilon_{\widetilde \eta^\leftarrow(T^{-1}\Gamma_{\widehat d_k},\V_{\widehat d_k})\V_{\widehat d_k} 
  \mathbf 1_{\widetilde \eta^\leftarrow\left(T^{-1}\Gamma_{\widehat d_k},\V_{\widehat d_k}\right) \geq \varepsilon}}.
\end{align*}
An application of Lemma \ref{kvgz.pp} yields the convergence of the points
\begin{align*}
& \left( A_n \X_{\lfloor nT\rfloor -k+1:\lfloor nT\rfloor} 
\mathbf 1_{\Vert A_n \X_{\lfloor nT\rfloor -k+1:\lfloor nT\rfloor} \Vert \geq \varepsilon}\right)_{k \in \mathbb N}\\
& \quad\schwach \left( \widetilde \eta^\leftarrow(T^{-1}\Gamma_{\widehat d_k},\V_{\widehat d_k})\V_{\widehat d_k} 
  \mathbf 1_{\widetilde \eta^\leftarrow(T^{-1}\Gamma_{\widehat d_k},\V_{\widehat d_k}) \geq \varepsilon}\right)_{k \in \mathbb N}
\end{align*}  
in $(\mathbb R^d \backslash \mathbb K_\varepsilon^d)^\mathbb N$. The desired result follows by taking the limit as 
$\varepsilon \downarrow 0$ and an easy application of Theorem 4.2 in \cite{billingsley:1999}.
\end{proof}

\section{Convergence of associated point processes}

Now that the limit distribution of normalized residual order statistics is identified we can study the associated marked 
point processes 
\begin{align*}
  \sum_{k=1}^{\lfloor nT \rfloor} \varepsilon_{(b_n^{-1} T_k, A_n \X_k)},~
 \sum_{k=1}^{\lfloor nT \rfloor} \varepsilon_{(b_n^{-1} T_{k-1}, A_n \X_k)}.
\end{align*}
In the uncoupled case convergence results for this processes can be established with a continuous mapping approach using the 
time deformation defined in \cite[(8.29)]{resnick:2007}. Since the continuity of this time deformation demands the processes 
$\A(\cdot)$ and $D(\cdot)$ to have a.s. no common jumps, which is not necessarily fulfilled in the coupled case, 
this standard methods connot be applied in our case. So we use a sorting argument like in (\ref{umsortierung}).

\begin{lem} \label{konvergenzlemma.pp}
Let $(\tau_n)_{n \in \mathbb N}$, $(\Gamma_n)_{n \in \mathbb N}$ and $(\V_n)_{n \in \mathbb N}$ be as in \eqref{reihendarst.}. Then the convergence of the associated point processes 
\begin{align}
& \sum_{k=1}^{\lfloor nT \rfloor} \varepsilon_{(b_n^{-1} T_k, A_n \X_k )}
  \schwach \sum_{k\in \mathbb N} 
  \varepsilon_{\left(D(T\cdot \tau_k),\widetilde \eta^\leftarrow(T^{-1}\Gamma_k,\V_k)\V_k
\right)} \label{kvgz.pp.1}\\
& \sum_{k=1}^{\lfloor nT \rfloor} \varepsilon_{(b_n^{-1} T_{k-1}, A_n \X_k )}
  \schwach \sum_{k\in \mathbb N} 
  \varepsilon_{\left(D(T\cdot \tau_k-),\widetilde \eta^\leftarrow(T^{-1}\Gamma_k,\V_k)\V_k
\right)} \label{kvgz.pp.2}
\end{align}
holds in $M_p([0,\infty) \times [-\infty,\infty]^d \backslash \{0\})$ for every $T>0$, where $D(x-)$ denotes the left-hand 
limit of the process $D(\cdot)$ in $x$.
\end{lem}

\begin{proof}
 Choose $T>0$ arbitrary. We start by sorting the points of the associated point process
\begin{align*}
 \left(\sum_{k=1}^{\lfloor nT\rfloor} \varepsilon_{\left(b_n^{-1}T_k,A_n \X_k \right)} \right)
 = \left(\sum_{k=1}^{\lfloor nT\rfloor} 
\varepsilon_{\left(b_n^{-1}T_{d_{\lfloor nT\rfloor -k+1}},A_n \X_{\lfloor nT\rfloor -k+1:\lfloor nT\rfloor} \right)} \right).
\end{align*}
Again $(d_1,\ldots,d_{\lfloor nT\rfloor})$ denotes the antirank vector of the r.v. $(\X_1,\ldots,\X_{\lfloor nT\rfloor})$. 
The normalized 
residual order statistics $A_n \X_{n-k+1:\lfloor nT\rfloor}$ have already been studied in Theorem \ref{verm.lepage}. It 
remains to determine the limit distribution of 
\begin{align*}
b_n^{-1}T_{d_{\lfloor nT\rfloor -k+1}} = b_n^{-1}\sum_{l=1}^{d_{\lfloor nT\rfloor -k+1}} J_l
= b_n^{-1} \sum_{l=1}^{\left\lfloor \lfloor nT\rfloor \cdot \frac{d_{\lfloor nT\rfloor -k+1}}{\lfloor nT\rfloor}\right\rfloor} J_l.
\end{align*}
Since $\X_1,\ldots,\X_{\lfloor nT\rfloor}$ are i.i.d., $(n^{-1}d_k)_{k \in \mathbb N} \schwach (T\cdot\tau_k)_{k \in \mathbb N}$ 
suggests that the convergence 
\begin{align} \label{konv.zuf.summe}
(b_n^{-1}T_{d_{\lfloor nT\rfloor -k+1}})_{k \in \mathbb N} \schwach (D(T\cdot\tau_k))_{k \in \mathbb N}
\end{align}
holds in $[0,\infty)^\mathbb N$. But this result cannot be established with a traditional continuous mapping approach. The 
mapping $\pi_t:D([0,\infty),\mathbb R^d)\times \mathbb R^+ \rightarrow \mathbb R,~(x,t)\mapsto x(t)$ is only a.s. continuous 
if $x$ is a.s. continuous in $t$. Also classical transfer theorems, cf. \cite{finkelstein:1989}, \cite{gnedenko:1968}, are not 
helpful because they require independence of the summands and their quantity or a stochastic convergence of the 
normalized antirank vector, cf. \cite{kern:2004}. Since none of these conditions is fulfilled another approach is used. 

Let $(\widehat d_1,\ldots, \widehat d_{\lfloor nT\rfloor})$ denote the associated antirank vector of the waiting times 
$(J_1,\ldots,J_{\lfloor nT\rfloor})$. Since the joint convergence of the well-centered and normalized sequential partial 
sums to the process $(D(\cdot),\A(\cdot))$ holds, one can easily prove convergence of the normalized antirank vector 
\begin{align*}
(n^{-1}(d_{\lfloor nT\rfloor -k+1})_{k \in \mathbb N}, n^{-1}(\widehat d_{\lfloor nT\rfloor -k+1})_{k \in \mathbb N}) 
\schwach ((T\cdot\tau_i)_{i\in\mathbb N},(T\cdot \widehat \tau_i)_{i \in \mathbb N})
\end{align*}
where $(\widetilde \tau_n)_{n \in \mathbb N}$, $(\widehat \tau_n)_{n \in \mathbb N}$ denote two i.i.d. sequence of $\mathcal U(0,1)$-distributed r.v.. 
Note that the sequences $(\widetilde\tau_n)_{n \in \mathbb N}$ and $(\widehat \tau_n)_{n \in \mathbb N}$ are not independent in the 
coupled case. As a consequence of the convergence of the antirank vector, the convergence of indicator functions 
\begin{align*}
 \left(\mathbf 1_{n^{-1} \widehat d_{\lfloor nT\rfloor -j+1} \leq n^{-1} d_{\lfloor nT\rfloor -i+1}}\right)_{j \in \mathbb N} 
\schwach \left(\mathbf 1_{\widehat \tau_j \leq\widetilde\tau_i}\right)_{j \in \mathbb N}
\end{align*}
holds. An application of Basu's lemma, cf. \cite[Theorem 5.1.2]{lehmann:2005}, yields the independence 
\begin{align} \label{anwendung.basu}
 \left(\left(d_1,\ldots d_{\lfloor nT \rfloor}\right), \left(\widehat d_1,\ldots \widehat d_{\lfloor nT \rfloor}\right)\right)
\perp \left(J_{1:\lfloor nT \rfloor} , \ldots ,J_{\lfloor nT \rfloor:\lfloor nT \rfloor}\right)
\end{align}
for every fixed $n \in \mathbb N$, which proves that the joint convergence 
\begin{align*}
 \left(b_n^{-1} J_{\lfloor nT\rfloor -j+1:\lfloor nT\rfloor} \mathbf 1_{\widehat d_{\lfloor nT\rfloor -j+1} 
\leq d_{\lfloor nT\rfloor -i+1}}\right)_{j \in \mathbb N}
\schwach \left((\eta^{\alpha})^{-1}(T^{-1}\widehat \Gamma_j)\mathbf 1_{\widehat \tau_j \leq\widetilde\tau_i}\right)_{j \in \mathbb N}
\end{align*}
holds, where $(\eta^\alpha)^{-1}$ is the right-sided inverse of the L\'evy measure associated with $D(1)$ and 
$(\widehat \Gamma_n)_{n \in \mathbb N}$ denotes a distributional copy of the sequence $(\Gamma_n)_{n \in \mathbb N}$. 
Note that the sequences $(\Gamma_n)_{n \in \mathbb N}$ and $(\widehat \Gamma_n)_{n \in \mathbb N}$ are also not independent 
in the coupled case. Now summation verifies
\begin{align}
& \sum_{l\in \mathbb N} b_n^{-1} J_{\lfloor nT\rfloor -l+1:\lfloor nT\rfloor} 
\mathbf 1_{n^{-1}\widehat d_{\lfloor nT\rfloor -l+1} \leq n^{-1} d_{\lfloor nT\rfloor -i+1}} \\
&  = \sum_{l=1}^{\lfloor nT \rfloor} b_n^{-1} J_{\lfloor nT\rfloor -l+1:\lfloor nT\rfloor} 
\mathbf 1_{n^{-1}\widehat d_{\lfloor nT\rfloor -l+1} \leq n^{-1} d_{\lfloor nT\rfloor -i+1}}  \nonumber \\
&  = b_n^{-1} \sum_{l=1}^{\lfloor nT \rfloor} J_l \mathbf 1_{l \leq d_{\lfloor nT\rfloor -i+1} }
  = b_n^{-1} \sum_{l=1}^{d_{\lfloor nT\rfloor -i+1}} J_l  
\schwach \sum_{l\in \mathbb N} (\eta^{\alpha})^{-1}(T^{-1}\widehat \Gamma_l) 
  \mathbf 1_{T\cdot \widehat \tau_l\leq T\cdot\widetilde\tau_i}. \label{fkd.proz.d}
\end{align}
Using the Ferguson-Klass series representation of the process $D(\cdot)$, cf. \cite{feguson:1972}, one identifies the 
right-hand side in (\ref{fkd.proz.d}) as series representation of $D(T\cdot\widetilde\tau_i)$. Since the convergence of antiranks 
holds simultaneously, we have proven
\begin{align} \label{nachweis.kvgz.pp.1}
  \left(b_n^{-1}\sum_{l=1}^{\lfloor n \cdot \frac{d_{\lfloor nT\rfloor -i+1}}{n}\rfloor}J_l \right)_{i \in \mathbb N} 
\schwach \left(D(T \cdot {\widetilde\tau_i}) \right)_{i \in \mathbb N}
\end{align}
in $(\mathbb R^+)^\mathbb N$. Again by the joint convergence of the properly normalized and scaled sequential partial sums 
to the process $(D(\cdot),\A(\cdot))$, the independence (\ref{anwendung.basu}) and Theorem \ref{verm.lepage}, 
convergence of the points 
\begin{align*}
& \left(b_n^{-1} T_{d_{\lfloor nT \rfloor -i + 1}},A_n 
\X_{\lfloor nT \rfloor -i +1 :\lfloor nT \rfloor}\right)_{i \in \mathbb N} 
\schwach \left(D(T\cdot {\widetilde\tau_i}),
\widetilde \eta^\leftarrow(T^{-1}\Gamma_{\widehat d_i},\V_{\widehat d_i})\V_{\widehat d_i} \right)_{i \in \mathbb N}
\end{align*} 
holds. Since the mapping $x\mapsto \varepsilon_x$ is continuous, the convergence of points yields the convergence of 
point processes
\begin{align*}
\left(\varepsilon_{\left(b_n^{-1} T_{d_{\lfloor nT \rfloor -i + 1}},A_n 
\X_{\lfloor nT \rfloor -i +1 :\lfloor nT \rfloor}\right)}\right)_{i \in \mathbb N}
\schwach \left(\varepsilon_{\left(D(T\cdot {\widetilde\tau_i}),
\widetilde \eta^\leftarrow(T^{-1}\widetilde\Gamma_{\widehat d_i},\V_{\widehat d_i})\V_{\widehat d_i}\right)}\right)
_{i \in \mathbb N}
\end{align*}
in $M_p([0,\infty) \times [-\infty,\infty]^d\backslash \{0\})$. Now summation and an easy application of theorem 4.2 in 
\cite{billingsley:1999} yields 
\begin{align*}
 \sum_{k=1}^{\lfloor nT\rfloor} \varepsilon_{\left(b_n^{-1} T_{d_{\lfloor nT \rfloor -k + 1}},A_n 
\X_{\lfloor nT \rfloor -k +1 :\lfloor nT \rfloor}\right)}
\schwach \sum_{k\in\mathbb N} 
\varepsilon_{\left(D(T\cdot {\widetilde\tau_k}),
\widetilde \eta^\leftarrow(T^{-1}\widetilde\Gamma_{\widehat d_k},\V_{\widehat d_k})\V_{\widehat d_k}\right)}
\end{align*}
in $M_p([0,\infty) \times [-\infty,\infty]^d\backslash \{0\})$. Now we need to reverse the order of the points again. 
We introduce
\begin{align*}
 r_k(\omega) := 1+ \# \{ i \in \mathbb N: \widetilde \eta^\leftarrow (T^{-1} \Gamma_i(\omega),\V_i(\omega)) > 
 \widetilde \eta^\leftarrow (T^{-1} \Gamma_k(\omega),\V_k(\omega)) \}
\end{align*}
as inverse of $\widehat d_k$. Note that $r_k$ is well-defined for all $k \in \mathbb N$ by Remark \ref{remark.existenz}. Moreover, 
since $(\widetilde\tau_n)_{n \in \mathbb N}$ and $(r_n)_{n \in \mathbb N}$ are independent, an easy application of the desintegration 
formula shows, that $(\widetilde\tau_{r_n})_{n \in \mathbb N}$ is also i.i.d.\ and $\mathcal U(0,1)$ distributed. Since the convergence in \eqref{grund.konv.pp} towards \eqref{reihendarst.} can be proven 
with the same technique, the 
sequences $(\widetilde\tau_{r_n})_{n \in \mathbb N}$, $(\Gamma_n)_{n \in \mathbb N}$ and $(\V_n)_{n \in \mathbb N}$ are also independent. 
So we define $\tau_n:=\widetilde\tau_{r_n}$ for all $n \in \mathbb N$ and obtain
\begin{align*}
\sum_{k=1}^{\lfloor nT\rfloor} \varepsilon_{\left(b_n^{-1} T_k,A_n \X_k \right)}
&  = \sum_{k=1}^{\lfloor nT\rfloor} \varepsilon_{\left(b_n^{-1} 
  T_{d_{\lfloor nT \rfloor -k+1:\lfloor nT \rfloor}},A_n \X_{\lfloor nT \rfloor -k+1:\lfloor nT \rfloor} \right)}\\
& \schwach 
% \sum_{k \in \mathbb N} 
% \varepsilon_{\left(D(T\cdot \tau_k), \widetilde\eta^\leftarrow(T^{-1}\Gamma_{\widehat d_k},\V_{\widehat d_k})\V_{\widehat d_k}\right)}
\sum_{k \in \mathbb N} 
\varepsilon_{\left(D(T\cdot \tau_k), \widetilde\eta^\leftarrow(T^{-1}\Gamma_k,\V_k)\V_k\right)}.
\end{align*}
Hence we have proven (\ref{kvgz.pp.1}). In order to prove (\ref{kvgz.pp.2}) the limit distribution of 
\begin{align*}
b_n^{-1} T_{d_{\lfloor nT\rfloor -i+1}-1} = b_n^{-1} \sum_{k=1}^{d_{\lfloor nT\rfloor -i+1}-1} J_k 
% = b_n^{-1} \sum_{k=1}^{\left\lceil n \cdot \frac{d_{\lfloor nT\rfloor -i+1}-1}{n}\right\rceil} J_k
 = b_n^{-1}\sum_{k=1}^{\left\lceil n \cdot \frac{d_{\lfloor nT\rfloor -i+1}}{n}-1\right\rceil} J_k
\end{align*}
has to be analized. Denoting $G([0,\infty),\mathbb R^d)$ the space of all left-continuous functions with right-hand limits from 
$[0,\infty)$ to $\mathbb R^d$ one easily proves that the mapping $\mathcal T:D[0,\infty)\rightarrow G[0,\infty),~x(t) 
\mapsto x(t-)$ is Lipschitz-continuous with Lipschitz-constant one and hence continuous. Since 
\begin{align*}
\mathcal T\left( \sum_{k=1}^{\lfloor nt\rfloor} J_k \right) = \sum_{k=1}^{\lceil nt\rceil -1} J_k
\end{align*}
holds, the continuity of $\mathcal T$ suggests, that the convergence
\begin{align} \label{nachweis.kvgz.pp.2}
 \left(b_n^{-1} \sum_{k=1}^{\big \lceil n \frac{d_{\lfloor nT\rfloor -i+1}}{n} -1 \big \rceil} J_k \right)_{i \in \mathbb N} 
 \schwach \left(D(T\cdot {\widetilde\tau_i-}) \right)_{i \in \mathbb N}
\end{align}
holds. But since (\ref{nachweis.kvgz.pp.1}) could not be proven with the continuous mapping theorem, we use the above 
arguments again to obtain (\ref{nachweis.kvgz.pp.2}). We note that 
\begin{align}
& b_n^{-1}\sum_{l=1}^{d_{\lfloor nT\rfloor -i+1}-1} J_l = \sum_{l=1}^{\lfloor nT\rfloor} b_n^{-1} 
J_{\lfloor nT \rfloor -l+1:\lfloor nt \rfloor} 
  \mathbf 1_{\widehat d_{\lfloor nT \rfloor -l+1} \leq d_{\lfloor nT\rfloor -i+1}-1}  \nonumber \\
& = \sum_{l=1}^{\lfloor nT\rfloor} b_n^{-1} J_{\lfloor nT \rfloor -l+1:\lfloor nt \rfloor} 
  \mathbf 1_{n^{-1}\widehat d_{\lfloor nT \rfloor -l+1} < n^{-1} d_{\lfloor nT\rfloor -i+1}} 
 \schwach \sum_{l\in \mathbb N} \widehat \eta^{-1}(T^{-1}\widehat \Gamma_l) 
\mathbf 1_{T\cdot \widehat\tau_l<T\cdot\widetilde\tau_i} \label{fkd.d-}
\end{align}
holds. Again we identify the limit on the right-hand side in (\ref{fkd.d-}) as a series representation of $D(T\cdot\widetilde\tau_i-)$. As 
already stated this yields the desired result (\ref{kvgz.pp.2}), which completes the proof.
\end{proof}

\section{Scaling limits of coupled CTRWs}

In this section we are now able to identify the scaling limits of coupled CTRWs using the limit theorems for their associated 
point processes stated in Lemma \ref{konvergenzlemma.pp}. We need to introduce the set 
\begin{align*}
 \mathcal S := \{ T\in \mathbb R^+ : \mathbb P(D(T\cdot \tau_i)=T) =0 \text{ for all } i \in \mathbb N \}
\end{align*}
for technical reasons. Due to selfsimilarity, one can easily show that the equality 
\begin{align*}
 D(xt\cdot \tau_i) \stackrel{\mathcal D}{=} x^{1/\alpha} D(t\cdot\tau_i)
\end{align*}
holds for all $i \in \mathbb N$ and $x,t \in \mathbb R^+$. Hence the set $\mathcal S$ is dense in $\mathbb R^+$. 

\begin{thm} \label{theorem.kvgz.ctrws}
Let $E(t):=\inf\{{ x>0}:D(x)>t\}$ denote the hitting-time process associated with $D(\cdot)$. Then convergence of the 
backward- and forward-coupled CTRW 
\begin{align}
&\sum_{k=1}^{N_{tb_n}} (A_n \X_k - \mathbb E(A_n \X_k \mathbf 1_{\Vert A_n \X_k \Vert \leq \tau})) \nonumber \\
& \schwach \lim_{\varepsilon \downarrow 0}  \left(\sum_{D(T\cdot {\tau_k})\leq t} 
\left( \widetilde \eta^\leftarrow (T^{-1} 
\Gamma_k,\V_k)\V_k 
\mathbf 1_{ \widetilde \eta^\leftarrow (T^{-1} \Gamma_k,\V_k) \geq \varepsilon} \right) 
- E(t) \int_{\varepsilon \leq \Vert x \Vert \leq \tau}x ~d\eta(x)\right) \label{kvgz.1}
\end{align}
and
\begin{align}
&\sum_{k=1}^{N_{tb_n}+1} (A_n \X_k - \mathbb E(A_n \X_k \mathbf 1_{\Vert A_n \X_k \Vert \leq \tau})) \nonumber \\
& \schwach \lim_{\varepsilon \downarrow 0} \left( \sum_{D(T\cdot {\tau_k -})\leq t} \hspace*{-0.2cm}
\left( \widetilde \eta^\leftarrow (T^{-1} 
\Gamma_k,\V_k)\V_k  
\mathbf 1_{\widetilde \eta^\leftarrow (T^{-1} \Gamma_k,\V_k) \geq \varepsilon}\right) 
- E(t) \int_{\varepsilon \leq \Vert x \Vert \leq \tau}x ~d\eta(x) \right) \label{kvgz.2}
\end{align}
holds in $D([0,T],\mathbb R^d)$ for every $T \in \mathcal S$ and every $\tau >0$ such that $\eta(\mathbb S^{d-1}_\tau) =0$.
\end{thm}

\begin{proof}
Choose $T \in \mathcal S$ arbitrary.  
Similar to (\ref{einschr.abb}) we define another a.s. continuous restriction functional 
\begin{align*}
& \widetilde \pi' : M_p([0,\infty) \times [-\infty,\infty]^d\backslash \{0\}) 
  \rightarrow M_p([0,\infty) \times [-\infty,\infty]^d \backslash \mathbb K^d_\varepsilon),\\
& \widetilde \pi'(m) := m_{|[0,\infty)\times [-\infty,\infty]^d \backslash \mathbb K^d_\varepsilon} 
\end{align*}
for $\varepsilon >0$ such that $\eta(S^{d-1}_\varepsilon)=0$. Moreover we define the summation functional 
\begin{align*}
& \chi:M_p([0,\infty) \times [-\infty,\infty]^d \backslash \mathbb K^d_\varepsilon) \rightarrow D([0,T],\mathbb R^d)\\
& \chi\left(\sum_{k\in \mathbb N}\varepsilon_{(t_k,x_k)}\right)(t) = \left(\sum_{t_k\leq t} x_k\right)_{t\in [0,T]} \hspace{-0.2cm}.
\end{align*}
which is a.s. continuous in the point 
\begin{align*}
 \sum_{k \in \mathbb N} \varepsilon_{(D(T\cdot \tau_k), \widetilde\eta^\leftarrow (T^{-1}\Gamma_k,\V_k)\V_k 
 \mathbf 1_{\widetilde\eta^\leftarrow (T^{-1}\Gamma_k,\V_k)>\varepsilon})}
\end{align*}
for every $T \in \mathcal S$. 
A proof of the continuity of $\chi$ is given in \cite[Sec. 7.2.3]{resnick:2007} for the case $d=1$ and can easily be modified to hold for $d\geq1$. 
So we apply the a.s. continuous mapping $\chi \circ \pi$ to the associated point processes in Lemma \ref{konvergenzlemma.pp}. 
Considering the equality $\{T_n \leq t\} = \{N_t \geq n\}$ we receive
\begin{align}
& \chi \circ \pi\left(\sum_{k=1}^{\lfloor nT \rfloor} \varepsilon_{\left(b_n^{-1}T_k,A_n \X_k\right)} \right)(t) 
 = \left(\sum_{b_n^{-1} T_k \leq t} A_n \X_k \mathbf 1_{\Vert A_n\X_k\Vert \geq \varepsilon}\right)_{t \in [0,T]} \nonumber \\
& = \left(\sum_{k=1}^{N_{tb_n}} A_n \X_k \mathbf 1_{\Vert A_n\X_k\Vert \geq \varepsilon}\right)_{t \in [0,T]}
  \schwach  \chi \circ \pi\left(\sum_{k\in \mathbb N} \varepsilon_{(D(T\cdot \tau_k),
  \widetilde \eta^\leftarrow (T^{-1}\Gamma_k,\V_k)\V_k)} \right) \nonumber \\
& = \sum_{D(T\cdot \tau_k)\leq t} \widetilde \eta^\leftarrow (T^{-1}\Gamma_k,\V_k)\V_k 
  \mathbf 1_{\widetilde \eta^\leftarrow (T^{-1}\Gamma_k,\V_k)\geq \varepsilon}. \label{gl.1}
\end{align}
in $D([0,T],\mathbb R^d)$. To study the centering constants we use the convergence 
\begin{align*}
 \left(\frac{N_{tb_n}}{n}\right)_{t \geq 0} \schwach (E(t))_{t \geq 0}
\end{align*}
in $D([0,\infty),\mathbb R^+)$, proved in Corollary 3.4 of \cite{meerschaert:2004}. Considering (\ref{konv.levymaß}) this yields
\begin{align*}
& \left(\sum_{k=1}^{N_{tb_n}} \mathbb E\left(A_n \X_k \mathbf 1_{[\varepsilon,\tau]}(\Vert A_n \X_k\Vert)\right)\right)_{t \geq 0}
= \left(\frac{N_{tb_n}}{n} \int_{\varepsilon \leq \Vert x \Vert \leq \tau} x  ~nd\mathbb P^{A_n \X_1}(x)\right)_{t \geq 0}\\
& \schwach \left(E(t) \int_{\varepsilon \leq \Vert x \Vert \leq \tau} x~d\eta(x)\right)_{t \geq 0}
\end{align*} 
in $D([0,\infty),\mathbb R^+)$. Since the process $E(\cdot)$ has a.s. continuous sample paths, Theorem 4.1 in \cite{whitt:1980} 
allows us to put this and (\ref{gl.1}) together. We obtain
\begin{align*}
& \left(\sum_{k=1}^{N_{tb_n}} \left(A_n \X_k \mathbf 1_{\Vert A_n \X_k\Vert\geq \varepsilon} - 
  \mathbb E\left(A_n \X_k \mathbf 1_{\varepsilon \leq \Vert A_n \X_k \Vert \leq \tau}\right)\right)\right)_{t\in [0,T]} \\
& \schwach \left(\sum_{D(T\cdot \tau_k) \leq t} \left(\widetilde \eta^\leftarrow (\Gamma_k,\V_k)\V_k   
  \mathbf 1_{\widetilde \eta^\leftarrow ( T^{-1} \Gamma_k,\V_k) \geq \varepsilon}\right) - E(t) \int_{\varepsilon 
  \leq \Vert x \Vert \leq \tau} x ~d\eta(x)\right)_{t\in [0,T]}
\end{align*}
in $D([0,T],\mathbb R^d)$. Taking limits as $\varepsilon \downarrow 0$ this yields the desired result (\ref{kvgz.1}). By Theorem 4.2 of 
\cite{billingsley:1999} it remains to show
\begin{align*}
& \lim_{\varepsilon \downarrow 0} \limsup_{n \rightarrow \infty} ~\mathbb P\Bigg(\sup_{0 \leq t \leq T} \Bigg \Vert 
\sum_{k=1}^{N_{tb_n}}\left( A_n \X_k \mathbf 1_{\Vert A_n \X_k \Vert\geq \varepsilon} - \mathbb E\left(A_n \X_k 
\mathbf 1_{\varepsilon \leq \Vert A_n \X_k \Vert \leq \tau}\right) \right)\\
& - \sum_{k=1}^{N_{tb_n}} \left( A_n \X_k - \mathbb E\left(A_n \X_k \mathbf 1_{\Vert A_n \X_k \Vert \leq \tau} \right). 
\right)\Bigg \Vert \geq \delta \Bigg) = 0
\end{align*}
for all $\delta >0$. Using a version of the Kolmogorov-inequality for integrable stopping times given in the Appendix 
and the norm-inequality $\Vert \cdot \Vert \leq \Vert \cdot \Vert_1$ we obtain
\begin{align*}
& \mathbb P\Bigg(\sup_{0 \leq t \leq T} 
  \Bigg \Vert \sum_{k=1}^{N_{tb_n}}\left( A_n \X_k \mathbf 1_{\Vert A_n \X_k \Vert\geq \varepsilon} 
  - \mathbb E\left(A_n \X_k \mathbf 1_{\varepsilon \leq \Vert A_n \X_k \Vert \leq \tau}\right)\right)  \\
& - \sum_{k=1}^{N_{tb_n}} \left(A_n \X_k - \mathbb E\left(A_n \X_k \mathbf 1_{\Vert A_n \X_k \Vert \leq \tau}\right)\right)\Bigg 
  \Vert \geq\delta \Bigg)  \nonumber \\
& \leq \mathbb P\left( \max_{1\leq j\leq N_{Tb_n}}\left\Vert
   \sum_{k=1}^{j}  A_n \X_k \mathbf 1_{\Vert A_n \X_k \Vert< \varepsilon} 
  - \mathbb E\left(A_n \X_k \mathbf 1_{\Vert A_n \X_k \Vert < \varepsilon}\right) \right\Vert_1 \geq \delta\right) \\
& \leq \mathbb P\left(\sum_{i=1}^d \max_{1\leq j\leq N_{Tb_n}}
    \left|\sum_{k=1}^{j}  (A_n \X_k)^{(i)} \mathbf 1_{\Vert A_n \X_k \Vert< \varepsilon} 
  - \mathbb E\left((A_n \X_k)^{(i)} \mathbf 1_{\Vert A_n \X_k \Vert < \varepsilon}\right)\right|  \geq \delta\right) \\
& \leq \mathbb P\left(\bigcup_{i=1}^d \max_{1\leq j\leq N_{Tb_n}}
    \left|\sum_{k=1}^{j}  (A_n \X_k)^{(i)} \mathbf 1_{\Vert A_n \X_k \Vert< \varepsilon} 
  - \mathbb E\left((A_n \X_k)^{(i)} \mathbf 1_{\Vert A_n \X_k \Vert < \varepsilon}\right)\right| \geq \frac{\delta}{d}\right) \nonumber \\
& \leq \sum_{i=1}^d 
  \mathbb P\left( \max_{1 \leq j \leq N_{Tb_n}+1}  
  \left| \sum_{k=1}^{j}  (A_n \X_k)^{(i)} \mathbf 1_{\Vert A_n \X_k \Vert< \varepsilon} 
  - \mathbb E\left((A_n \X_k)^{(i)}\mathbf 1_{\Vert A_n \X_k \Vert < \varepsilon}\right)\right| \geq \frac{\delta}{d}\right) \nonumber \\
& \leq \left(\frac{\delta}{d}\right)^{-2}  \mathbb E\left(N_{Tb_n}+1 \right) \sum_{i=1}^d 
   Var\left((A_n \X_1)^{(i)} \mathbf 1_{\Vert A_n \X_1 \Vert < \varepsilon} \right)  \\
& \leq \left(\frac{\delta}{d}\right)^{-2} 
\mathbb E\left(N_{Tb_n}+1 \right) \sum_{i=1}^d 
  \mathbb E\left(\left((A_n \X_1)^{(i)} \mathbf 1_{\Vert A_n \X_1 \Vert < \varepsilon} 
  \right)^2\right),
\end{align*}
where $x^{(n)}$ denotes the $n$-th coordinate of the vector $x$. Note that we have to take $N_{tb_n}+1$ since $N_{tb_n}$ does 
not fulfill the conditions of Lemma \ref{kolmogoroff.ungl}. 
Now Theorem 9 in \cite{mallor:2006} states that $\mathbb E\left(N_t+1 \right)$ can asymptotically be expressed by the integrated tail
of the distribution function of $J_1$
\begin{align*}
 \mathbb E\left(N_t+1 \right) \Gamma(2-\alpha)\Gamma(1+\alpha) \sim \frac{t}{\int_0^t (1-F_{J_1}(x))dx},
\end{align*}
where $\Gamma$ denotes the gamma-function. By Karamata's Theorem the limit behavior of the function 
$t \mapsto \int_0^t (1-F_{J_1}(s))ds$ can be expressed by $F_{J_1}$
\begin{align*}
 \int_0^t (1-F_{J_1}(s))ds \sim \frac{t(1-F_{J_1}(t))}{1-\alpha}.
\end{align*}
Putting this together we obtain
\begin{align*}
 \mathbb E\left(N_t+1 \right) \sim (1-F_{J_1}(t))^{-1} \cdot \Gamma(1-\alpha)^{-1} \cdot \Gamma(1+\alpha)^{-1}.
\end{align*}
Defining $C:=(\Gamma(1-\alpha) ~\Gamma(1+\alpha) ~\eta^\alpha((T,\infty)))^{-1}$ for abbreviation, the inequality
$|x^{(n)}| \leq \Vert x\Vert$ yields
\begin{align*}
& \lim_{\varepsilon \downarrow 0} \limsup_{n \rightarrow \infty} 
\left(\frac{\delta}{d}\right)^{-2} \lim_{\varepsilon \downarrow 0} \limsup_{n \rightarrow \infty} 
\mathbb E\left(N_{Tb_n}+1 \right) \sum_{i=1}^d 
  \mathbb E\left(\left((A_n \X_1)^{(i)} \mathbf 1_{\Vert A_n \X_1 \Vert < \varepsilon} 
  \right)^2\right) \\
& \leq \lim_{\varepsilon \downarrow 0} \limsup_{n \rightarrow \infty} 
  \left(\frac{\delta}{d}\right)^{-2} C
 \cdot  \sum_{i=1}^d \lim_{\varepsilon \downarrow 0} \limsup_{n \rightarrow \infty} 
  n \mathbb E\left(((A_n \X_1)^{(i)})^2 \mathbf 1_{|(A_n \X_1)^{(i)}| < \varepsilon}\right).
\end{align*}
So it remains to show
\begin{align*}
\lim_{\varepsilon \downarrow 0} \limsup_{n \rightarrow \infty} 
n \mathbb E\left(((A_n \X_1)^{(i)})^2 \mathbf 1_{|(A_n \X_1)^{(i)}| < \varepsilon}\right)=0 
\end{align*}
for all $1 \leq i \leq d$. Defining the tail and truncated second moment of $\X_1$ in direction $v$ by 
\begin{align*}
 V(r,v) := \mathbb P(|\langle\X_1,v\rangle| > r), ~U(r,v) := \mathbb E(\langle\X_1,v\rangle^2 \mathbf 1_{|\langle\X_1,v\rangle| < r})
\end{align*}
for every $v \in \mathbb S^{d-1}$ and every $r>0$ an easy calculation shows that 
\begin{align*}
 \mathbb E\left(((A_n \X_1)^{(i)})^2 \mathbf 1_{|(A_n \X_1)^{(i)}| < \varepsilon}\right) 
= r_n^2 U(r_n^{-1} \varepsilon, v_n)
\end{align*}
holds. Here $r_n>0$ and $v_n \in \mathbb S^{d-1}$ are taken such that $A_n^* e_i = r_n v_n$ holds for every $n \in \mathbb N$, 
where $A^*$ is the adjoint of $A$ and $e_1,\ldots,e_d$ denotes the standard basis of $\mathbb R^d$. With this notation we have 
to analyse 
\begin{align} \label{3faktoren}
& n \cdot r_n^2 U\left( r_n^{-1} \varepsilon, v_n\right) 
 = \varepsilon^2 \frac{U(r_n^{-1} \varepsilon,v_n)}{\varepsilon^2 r_n^{-2} 
  V\left(r_n^{-1} \varepsilon ,v_n\right)} \cdot \frac{V\left(r_n^{-1} \varepsilon,v_n\right)}
  {V\left(\varepsilon^{-1}( r_n^{-1}\varepsilon) ,v_n \right)} \cdot n \cdot V(r_n^{-1} ,v_n).
\end{align}
One easily verifies that the third factor in (\ref{3faktoren}) is bounded by $\eta(\{ x\in \mathbb R^d :|x^{(i)}| > 1 \})$. 
Since the real parts $a_1 \leq \ldots \leq a_d$ of all eigenvalues of the operator $E$ are greater than $1/2$ we can find 
$\widetilde \varepsilon >0$ such that $2-\widetilde \varepsilon-a_1^{-1}>0$ holds. So an application of Theorem 6.3.4 of 
\cite{meerschaert:2001} yields the existence of a constant $C_1$, such that for $n \in \mathbb N$ large 
enough the second factor in (\ref{3faktoren}) is bounded by $C_1 \varepsilon^{2-\widetilde \varepsilon-a_1^{-1}}$. Finally, Corollary 6.3.9 in \cite{meerschaert:2001} yields that the first factor in (\ref{3faktoren}) is bounded by a constant $C_2$. Putting things together this
completes the proof.

The proof of the convergence \eqref{kvgz.2} works the same way.
\end{proof}

Theorem \ref{theorem.kvgz.ctrws} provides a series representation for the limit distribution. This representation might be useful for 
simulation purposes. Now it is of considerable interest to identify the scaling limits with the ones stated in \cite{henry:2010}. Let $\A(t-)^+$ denote the right-continuous version of the process $\A(t-)$, i.e. $\A(t-)^+$ is an element of $D([0,\infty),\mathbb R^d)$. Using a continuous mapping approach on $D([0,\infty),\mathbb R^d)$, Straka and Henry \cite{henry:2010} show that the limit laws in \eqref{kvgz.1} and \eqref{kvgz.2} coincide with the distribution of $\A(E(t)-)^+$, respectively $\A(E(t))$. More precisely, Straka and Henry follow a more general approach and consider triangular arrays instead of i.i.d.\ sequences following \cite{meerschaert:2008} and they determine the joint scaling limit of the time process together with the CTRW. Note that for an uncoupled CTRW, where waiting times and jumps are independent, the processes $(\A(t))_{t\geq0}$ and $(E(t))_{t\geq0}$ are independent and thus the two limiting processes $(\A(E(t)-)^+)_{t\geq0}$ and $(\A(E(t)))_{t\geq0}$ coincide in distribution. This equality fails for a coupled CTRW and differences in the two limiting processes are illustrated in \cite{jurlewicz:2010}, where the limiting distributions, their Fourier-Laplace transforms and the corresponding governing pseudo differential equations for their densities are given in terms of the joint distribution of $(D(t),\A(t))_{t\geq0}$ and its L\'evy exponent.

Our arguments for the identification of the scaling limits are based on the following equalities:
\begin{align} \label{mengengleichheiten}
 \{D(x)<t\} = \{x<E(t)\},~\{D(x-)\leq t\} = \{x \leq E(t)\}. 
\end{align}
The left-hand side of (\ref{mengengleichheiten}) is already proven in $(3.2)$ of \cite{meerschaert:2004}. For the proof of the 
right-hand side assume $D(x-)\leq t$ holds. So we have $D(y) \leq t$ for all $y<x$. Hence $x\leq E(t)$. Otherwise if 
$D(x-)>t$ holds, there exists an $\varepsilon >0$ such that $D(y)>t$ holds for all $y\geq x-\varepsilon$. Hence 
$E(t) \leq x-\varepsilon <x$.

\begin{cor}
The convergence 
 \begin{align} \label{reihendarst.1}
  \sum_{k=1}^{N_{tb_n}} (A_n \X_k - \mathbb E(A_n \X_k \mathbf 1_{\Vert A_n \X_k\Vert \leq \tau})) \schwach \A(E(t)-)^+
 \end{align}
and
\begin{align} \label{reihendarst.2}
 \sum_{k=1}^{N_{tb_n}+1} (A_n \X_k - \mathbb E(A_n \X_k 1_{\Vert A_n \X_k\Vert \leq \tau})) \schwach \A(E(t))
\end{align}
holds in $D([0,T],\mathbb R^d)$ for every $T \in \mathcal S$ and every $\tau >0$ such that $\eta(\mathbb S^{d-1}_\tau)=0$. 
\end{cor}

\begin{proof}
 The convergence (\ref{reihendarst.2}) can easily be verified applying the right-hand side of (\ref{mengengleichheiten}) 
to the series representation (\ref{reihendarst.}). To prove (\ref{reihendarst.1}) we apply the left-hand side of 
(\ref{mengengleichheiten}) to (\ref{kvgz.1}) and obtain
\begin{align}
&\sum_{k=1}^{N_{tb_n}} (A_n \X_k - \mathbb E_\tau(A_n \X_k)) \nonumber \\
& \schwach \lim_{\varepsilon \downarrow 0}  \left(\sum_{D(T\cdot \tau_k)\leq t} 
\left( \widetilde \eta^\leftarrow (T^{-1} 
\Gamma_k,\V_k)\V_k 
\mathbf 1_{ \widetilde \eta^\leftarrow (T^{-1} \Gamma_k,\V_k) \geq \varepsilon} \right) 
- E(t) \int_{\varepsilon \leq \Vert x \Vert \leq \tau}x ~d\eta(x)\right) \nonumber \\
& = \lim_{\varepsilon \downarrow 0}  \Bigg(\sum_{T\cdot  \tau_k< E(t)} 
\left( \widetilde \eta^\leftarrow (T^{-1} 
\Gamma_k,\V_k)\V_k 
\mathbf 1_{ \widetilde \eta^\leftarrow (T^{-1} \Gamma_k,\V_k) \geq \varepsilon} \right) \nonumber \\
& + \sum_{D(T\cdot \tau_k)=t} 
\left( \widetilde \eta^\leftarrow (T^{-1} 
\Gamma_k,\V_k)\V_k 
\mathbf 1_{ \widetilde \eta^\leftarrow (T^{-1} \Gamma_k,\V_k) \geq \varepsilon} \right) 
- E(t) \int_{\varepsilon \leq \Vert x \Vert \leq \tau}x ~d\eta(x) \Bigg). \label{hilfsgl.1}
\end{align}
in $D([0,T],\mathbb R^d)$. As we already stated
\begin{align*}
\lim_{\varepsilon \downarrow 0}  \Bigg(\sum_{T\cdot \tau_k< E(t)} 
\left( \widetilde \eta^\leftarrow (T^{-1} 
\Gamma_k,\V_k)\V_k 
\mathbf 1_{ \widetilde \eta^\leftarrow (T^{-1} \Gamma_k,\V_k) \geq \varepsilon} \right) 
- E(t) \int_{\varepsilon \leq \Vert x \Vert \leq \tau}x ~d\eta(x) \Bigg) \\
\end{align*}
is a series representation of $\A(E(t)-)$. The extra summands only have to be considered if a jump occurs at a time $t$ with $D(T\cdot \tau_k )=t$. This yields the right-continuity of the 
limit and we have proven (\ref{reihendarst.1}).
\end{proof}

 \appendix

\section{A generalization of Kolmogorov's inequality}

The following generalization of Kolmogorov's inequality can be shown by standard techniques. However, we were not able 
to find a suitable proof in the literature and will only give a sketch of proof. 

\begin{lem} \label{kolmogoroff.ungl}
 Let $(Y_n)_{n \in \mathbb N}$ be i.i.d. with $\mathbb E(Y_1)=0$ and $T$ be an $\mathbb N_0$-valued integrable stopping time with 
respect to the filtration $\mathcal F_n := \sigma(Y_1,\ldots,Y_n)$. Then 
\begin{align*}
 \mathbb P\left(\max_{1 \leq k \leq T} \left|\sum_{j=1}^k Y_j\right|\geq \delta \right) 
 \leq \delta ^{-2} \cdot \mathbb E(T) \cdot Var(Y_1)
\end{align*}
holds for alle $\delta >0$. 
\end{lem}

\begin{proof}
 First we restrict our attention to the truncated stopping time $T \wedge n$. An easy calculation shows that 
\begin{align*}
 \mathbb E \left(\sum_{k=1}^{T\wedge n}\left(S_k - S_{k-1} \right)^2 \right) 
= \mathbb E(S_{T\wedge n}^2)
\end{align*}
holds, where $S_n := \sum_{k=1}^n Y_k$ denotes the $n$-th partial sum. Defining 
\begin{align*}
M_0 := 0,~ M_{k+1} := \begin{cases}
                       S_{k+1}, & \text{ if } \max_{1 \leq j \leq k} \left|S_j\right| < \delta\\
		       M_k, & \text{ else}
                      \end{cases}
\end{align*}
the same calculation shows
\begin{align*}
 \mathbb E \left(\sum_{k=1}^{T\wedge n} \left(M_k - M_{k-1} \right)^2 \right)
 = \mathbb E(M_{T\wedge n}^2) - 2 \cdot \mathbb E\left(\sum_{k \in \mathbb N} (M_k - M_{k-1}) S_{k-1} 
  \mathbf 1_{k \leq T\wedge n} \right).
\end{align*}
Using the definition of $M_n$ one can show
\begin{align*}
 \mathbb E \left(\sum_{k=1}^{T\wedge n}\left(M_k - M_{k-1} \right)^2 \right)
= \mathbb E(M_{T\wedge n}^2).
\end{align*}
Now $\left| M_k - M_{k-1} \right| \leq \left| S_k - S_{k-1} \right|$ holds for all $k \in \mathbb N$. Hence the Markov-inequality 
yields 
\begin{align*}
& \mathbb P \left(\max_{1 \leq k \leq T\wedge n} \left|S_k\right| \geq \delta \right)= \mathbb P\left( \left| M_{T\wedge n}
  \right|\geq\delta\right) 
  \leq \delta^{-2} ~ \mathbb E\left(M_{T\wedge n}^2 \right) \\
&  = \delta^{-2} ~ \mathbb \mathbb E \left(\sum_{k=1}^{T\wedge n} \left(M_k - M_{k-1} \right)^2 \right)  
 \leq \delta^{-2} ~ \mathbb \mathbb E \left(\sum_{k=1}^{T\wedge n}\left(S_k - S_{k-1} \right)^2 \right) 
  = \delta^{-2} ~ \mathbb E\left(S_{T\wedge n}^2 \right). 
\end{align*}
An application of Wald's inequality yields the desired result for $T\wedge n$. The generalization for the stopping time $T$ 
follows by the martingale convergence theorem. 
\end{proof}

\begin{rem}
 Let $(Z_n)_{n \in \mathbb N}$ be a sequence such that $(Y_n,Z_n)_{n \in \mathbb N}$ is i.i.d.. Then Lemma \ref{kolmogoroff.ungl} 
also holds replacing $\mathcal F_n$ by $\widetilde{\mathcal F_n} := \sigma((Y_1,Z_1),\ldots,(Y_n,Z_n))$. 
\end{rem}


%% References


% \bibliography{<your-bib-database>}



\end{document}